\RequirePackage{fix-cm}

\documentclass[nonatbib]{elsarticle}   

\usepackage{amsmath,amssymb,amsthm,latexsym,soul,cite,mathrsfs}
\usepackage[usenames,dvipsnames]{color}

\usepackage{enumitem,graphicx}
\usepackage[colorlinks=true,urlcolor=blue,
citecolor=red,linkcolor=blue,linktocpage,pdfpagelabels,
bookmarksnumbered,bookmarksopen]{hyperref}
\usepackage[english]{babel}

\usepackage[hyperpageref]{backref}

\usepackage[normalem]{ulem}

\numberwithin{equation}{section}
\newtheorem{theorem}{Theorem}[section]
\newtheorem{proposition}[theorem]{Proposition}
\newtheorem{lemma}[theorem]{Lemma}

\newtheorem{remark}{Remark}

\newcommand\R{\mathbb R}

\makeatletter
\newcommand{\aslabel}[1]{#1\def\@currentlabel{#1}}
\newcommand{\nowlabel}[1]{\def\@currentlabel{#1}}
\makeatother

\newcommand{\rob}[1]{\left({#1}\right)}
\newcommand{\sqb}[1]{\left[{#1}\right]}
\newcommand{\cub}[1]{\left\{{#1}\right\}}
\newcommand{\n}[1]{\left\|{#1}\right\|}
\newcommand{\m}[1]{\left|{#1}\right|}

\usepackage[latin1]{inputenc}

\begin{document}

\journal{...}
\title{A nonsmooth variational approach to semipositone quasilinear problems in $\mathbb{R}^N$}


\author[1]{Jefferson Abrantes  Santos\fnref{t1}}  \address[1]{Universidade Federal de Campina Grande,  Unidade Acad\^emica de Matem\'atica, \\  CEP: 58429-900, Campina Grande - PB, Brazil.}  \ead{jefferson@mat.ufcg.edu.br}    
\author[1]{Claudianor O. Alves\fnref{t2}} \ead{coalves@mat.ufcg.edu.br}    
\author[3]{Eugenio Massa\fnref{t3}}\address[3]{{ Departamento de Matem\'atica,
                     Instituto de Ci\^encias Matem\'aticas e de Computa\c c\~ao, Universidade de S\~ao Paulo,
                     Campus de S\~ao Carlos, 13560-970, S\~ao Carlos SP, Brazil.}}
                  \ead{eug.massa@gmail.com}

\fntext[t1]{J. Abrantes Santos was partially supported by CNPq/Brazil 303479/2019-1}\fntext[t2]{C. O. Alves was partially supported by CNPq/Brazil 304804/2017-7}\fntext[t3]{E. Massa was partially  supported by grant $\#$303447/2017-6, CNPq/Brazil.}


\date{}

\begin{abstract}
	This paper concerns the existence of a solution for the following class of semipositone quasilinear problems 
	\begin{equation*}
		\left \{
		\begin{array}{rclcl}
			-\Delta_p u & =  & h(x)(f(u)-a) & \mbox{in} & \mathbb{R}^N, \\
			u& > & 0 & \mbox{in} & \mathbb{R}^N, \\
		\end{array}
		\right.
	\end{equation*}
where $1<p<N$, $a>0$, $ f:[0,+\infty) \to [0,+\infty)$ is a  function with subcritical growth and $f(0)=0$, while $h:\mathbb{R}^N \to (0,+\infty)$ is a continuous function that satisfies some technical conditions. 
We prove via nonsmooth critical points theory and comparison principle,  that a solution exists for $a$ small enough. We also provide a version of Hopf's Lemma  and a Liouville-type result for the $p$-Laplacian in the whole $\mathbb{R}^N$.
\medskip

\noindent  {\bf  Mathematical Subject Classification MSC2010:} 
35J20,   
35J62   
(49J52).   

\end{abstract}

\maketitle	
\section{Introduction}
In this paper we study the existence of positive weak solutions for the $p$-Laplacian semipositone problem in the whole space 
\begin{equation}\label{Problem-P}\tag{$P_a$}
	\left \{
	\begin{array}{rclcl}
	-\Delta_p u & =  & h(x)(f(u)-a) & \mbox{in} & \mathbb{R}^N, \\
	u& > & 0 & \mbox{in} & \mathbb{R}^N, \\
	\end{array}
	\right.
\end{equation}
where $1<p<N$, $a>0$, $f:[0,+\infty) \to [0,+\infty)$ is a continuous function with subcritical growth and $f(0)=0$. Moreover, the function $h:\mathbb{R}^N \to (0,+\infty)$ is a continuous function satisfying  
\begin{itemize}
\item[(\aslabel{$P_{1}$})\label{Hp_P2_L1Li}]   $h\in L^{1}(\mathbb{R}^N)\cap L^{\infty}(\mathbb{R}^N)$, 
\item[(\aslabel{$P_{2}$})\label{Hp_P4_Bbeta}]  $h(x)<B|x|^{-\vartheta}$ for $x\neq0$, with $\vartheta>N$ and $B>0$.
\end{itemize}

An example of a function $h$ that satisfies the hypotheses \eqref{Hp_P2_L1Li}$-$\eqref{Hp_P4_Bbeta} is given below:
$$
h(x)=\frac{B}{1+|x|^{\vartheta}}, \quad \forall x \in \mathbb{R}^N.
$$

In the whole of this paper, we say that a function $u \in D^{1,p}(\mathbb{R}^N)$ is a weak solution for (\ref{Problem-P}) if
$u$ is a continuous positive function that verifies
$$
\int_{\mathbb{R}^N}|\nabla u|^{p-2}\nabla u \nabla v\,dx=\int_{\mathbb{R}^N}h(x)(f(u)-a)v\,dx, \quad \forall v \in D^{1,p}(\mathbb{R}^N).
$$

 \subsection{State of art.} 
The problem (\ref{Problem-P}) for $a = 0$ is very simple to be solved, either employing  the well known
mountain pass theorem due to Ambrosetti and Rabinowitz \cite{AmbRab1973}, or via minimization. However, for the
case where (\ref{Problem-P}) is semipositone, that is, when $a> 0$, the existence of a positive
solution is not so simple, because the standard arguments via the mountain pass
theorem combined with the maximum principle do not directly give a positive
solution for the problem, and in this case, a very careful analysis must be done.

The literature associated with semipositone problems in bounded domains is very rich since the appearance
of the paper by Castro and Shivaji \cite{CasShi1988} who were the first to consider this
class of problems. We have observed that there are different methods to prove the existence
and nonexistence of solutions, such as subsupersolutions, degree theory arguments,
fixed point theory and bifurcation; see for example \cite{AliCaShi}, \cite{AmbArcBuff},\cite{AllNisZecca}, \cite{AnuHaiShi1996} and their
references. In addition to these methods, also variational methods were  used in a
few papers as can be seen in \cite{AldHoSa19_semipos_omega}, \cite{CalCasShiUns2007}, \cite{CFELo}, \cite{CDS}, \cite{CQT}, \cite{CTY}, \cite{DC}, \cite{FigMasSan_KirchSemipos}, \cite{MR2008685} and \cite{Jea1_cont}. We would like to point out that in  \cite{CFELo}, Castro, de Figueiredo and Lopera studied the existence of solutions for the following class of semipositone quasilinear problems
\begin{equation}\label{Problem-P2}
	\left \{
	\begin{array}{rclcl}
		-\Delta_p u & =  & \lambda f(u) & \mbox{in} & \Omega, \\
		u(x)& > & 0 & \mbox{in} & \Omega, \\
		u & =  & 0 & \mbox{on} & \partial\Omega, \\
	\end{array}
	\right.
\end{equation}
where $\Omega \subset\mathbb{R}^{N}$, $N > p>2$, is a smooth bounded domain, $\lambda >0$ and $f:\mathbb{R} \to \mathbb{R}$ is a differentiable function with $f(0)<0$. In that paper, the authors assumed that there exist $q \in (p-1, \frac{Np}{N-p}-1), A,B>0 $ such that
$$
\left\{
\begin{array}{l}
	A(t^q-1)\leq f(t) \leq B(t^q-1), \quad \mbox{for} \quad t>0\\
	f(t)=0, \quad \mbox{for} \quad  t \leq -1.
\end{array}
\right.
$$
The existence of a solution was proved by combining the mountain pass theorem with the regularity theory. Motivated by the results proved in \cite{CFELo}, Alves, de Holanda and dos Santos \cite{AldHoSa19_semipos_omega} studied the existence of solutions for a large class of semipositone quasilinear  problems of the type
\begin{equation}\label{Problem-P3}
	\left \{
	\begin{array}{rclcl}
		-\Delta_\Phi u & =  & f(u)-a & \mbox{in} & \Omega, \\
		u(x)& > & 0 & \mbox{in} & \Omega, \\
		u & =  & 0 & \mbox{on} & \partial\Omega, \\
	\end{array}
	\right.
\end{equation}
where $\Delta_\Phi$ stands for the $\Phi$-Laplacian operator. The proof of the main result is also done via variational methods, however in their approach the regularity results
found in Lieberman \cite{L1,L2} play an important role. By using the mountain pass theorem, the authors found a
solution $u_a$ for all $a>0$. After that, by taking the limit when  $a$  goes to 0 and using the regularity results  in \cite{L1,L2}, they proved that $u_a$ is positive for $a$ small enough.

Related to semipositone problems in unbounded domains, we only found the paper due to Alves, de Holanda, and dos Santos \cite{AldHoSa19_semipos_RN} that studied the existence of solutions for the following class of problems
\begin{equation}\label{Problem-PAHS}
	\left \{
	\begin{array}{rclcl}
		-\Delta u & =  & h(x)(f(u)-a) & \mbox{in} & \mathbb{R}^N, \\
		u& > & 0 & \mbox{in} & \mathbb{R}^N, \\
	\end{array}
	\right.
\end{equation}
where $a>0$, $f:[0,+\infty) \to [0,+\infty)$ and $h:\mathbb{R}^N \to (0,+\infty)$ are continuous functions with $f$ having a subcritical growth and $h$ satisfying some technical conditions. The main tools used were  variational methods combined with the Riesz potential theory.

 \subsection{Statement of the main results.} 
Motivated by the results found in \cite{CFELo}, \cite{AldHoSa19_semipos_omega} and \cite{AldHoSa19_semipos_RN}, we intend to study the existence of solutions for (\ref{Problem-P}) with two different types of nonlinearities. In order to state our first main result, we assume the following conditions on $f$: 
\begin{itemize}
	\item[(\aslabel{$f_0$}) \label{Hp_forig}] \qquad $$\displaystyle\lim_{t\to 0^+} \frac{F(t)}{t^{p}}=0\,;$$ 
\end{itemize} 
\begin{itemize}
\item[(\aslabel{$f_{sc}$})\label{Hp_estSC}] \qquad there exists $q\in (1,p^*)$ such that 
$\displaystyle \limsup_{t\to +\infty} \frac{f(t)}{t^{q-1}}<\infty,$ 
\end{itemize} 
where   $p^*=\frac{pN}{N-p}$ is the critical Sobolev exponent;
\begin{itemize}
\item[(\aslabel{$f_\infty$}) \label{Hp_PS_SQ}] \qquad 
$q>p$  in \eqref{Hp_estSC} and there exist $\theta>p$ and $t_0>0$ such that 
\begin{eqnarray*}
	&& 
0<	\theta    F(t) \leq f(t)t, \quad \forall t>t_0,
\end{eqnarray*}
\end{itemize}
where $F(t)=\int_{0}^{t}f(\tau) \, d \tau$.	

Our first main result has the following statement

\begin{theorem} \label{Theorem1} Assume the conditions  \eqref{Hp_P2_L1Li}$-$\eqref{Hp_P4_Bbeta}, \eqref{Hp_estSC}, \eqref{Hp_forig} and \eqref{Hp_PS_SQ}.  
	Then there exists $a^{\ast}>0$ 
	such that, if $a\in[0,a^{\ast})$,  problem \eqref{Problem-P} has a positive weak solution $u_a\in C(\mathbb{R}^N) \cap D^{1,p}(\mathbb{R}^N)$. 
\end{theorem}

As mentioned above, a version of Theorem \ref{Theorem1} was proved in  \cite{AldHoSa19_semipos_RN} in the semilinear case $p=2$. Their proof exploited  variational methods for $C^1$ functionals and Riesz potential theory in order to prove the positivity of the solutions of a smooth approximated problem, which then resulted to be  actual solutions of problem \eqref{Problem-P}. In our setting, since we are working with the $p$-Laplacian, that is a nonlinear operator, we do not have a Riesz potential theory analogue that works well for this class of operator. Hence, 
a different approach was developed in order to treat the problem (\ref{Problem-P}) for $p \not =2$. Here, we make a different approximation for problem \eqref{Problem-P}, which results in working with a {\it nonsmooth approximating functional}. 

As a result, the Theorem \ref{Theorem1} is also new when $p=2$, since the set of hypotheses we assume here is different. In fact, avoiding the use of the Riesz theory, we do not need to assume that $f$ is Lipschitz (which would not even be possible in the case of condition \eqref{Hp_forig_up} below), and a different condition on the decaying of the function $h$ is required.

 The use of the nonsmooth approach turns out to simplify several technicalities usually involved in the treatment of semipositone problems. Actually,  working with the  $C^1$ functional naturally associated to \eqref{Problem-P}, one obtains critical points $u_a$ that may be negative somewhere. 
 When working in bounded sets, the positivity of $u_a$ is obtained, in the limit as $a\to0$, by proving  convergence in $C^1$ sense to the positive solution $u_0$ of the case $a=0$, which is enough since $u_0$ has also normal derivative at the boundary which is bounded away from zero in view of the Hopf's Lemma. This approach can be seen for instance in \cite{Jea1_cont,CFELo,AldHoSa19_semipos_omega,FigMasSan_KirchSemipos}.
In $\R^n$, a different argument must be used: actually one can obtain convergence on compact sets, but the limiting solution  $u_0$ goes to zero at infinity as $|x|^{(p-N)/(p-1) }$ (see Remark \ref{rm_udec}), which means that one needs to be able to do some finer estimates on the convergence. 
  In \cite{AldHoSa19_semipos_RN}, with $p=2$,   the use of the Riesz potential, allowed to prove  that  $|x|^{N-2 }|u_a-u_0|\to0$ uniformly, which then led to the positivity of $u_a$ in the limit. 
  
In the lack of this tool, we had to find a different way to   prove the positivity of $u_a$. 
The great advantage of our approach via nonsmooth analysis, is that our critical points $u_a$ will always be nonnegative functions (see Lemma \ref{lm_prop_minim}). In spite of not necessary being week solutions of the equation in problem   \eqref{Problem-P},  they  turn out to be supersolutions  and also subsolutions of the limit equation with $a=0$. These properties will allow us to use comparison principle in order to prove the strict positivity of $u_a$ with the help of a suitable  barrier function (see the Lemmas \ref{lm_z} and \ref{lm_ujpos}). From the positivity it will immediately follow that $u_a$ is indeed a weak  solutions of \eqref{Problem-P}.

\par\medskip

The reader is invited to see that by \eqref{Hp_PS_SQ}, there exist $A_1,B_1>0$ such that
\begin{equation} \label{AR}
	F(t) \geq A_1|t|^{\theta}-B_1, \quad for \  t\geq0.
\end{equation}
This inequality yields that the functional we will be working with is not bounded from below.		
On the other hand, the condition \eqref{Hp_forig} will produce a ``range of mountains" geometry around the origin for the functional, which  completes the mountain pass structure.  
Finally, conditions  \eqref{Hp_estSC} and \eqref{Hp_PS_SQ} impose a subcritical growth to $f$, which are used to obtain the required compactness condition.

\par\medskip
	
Next, we are going to state our second result. For this result, we still assume \eqref{Hp_estSC} together with the following conditions:  
\begin{itemize}
\item[(\aslabel{$\widetilde f_0$}) \label{Hp_forig_up}]  $$\displaystyle\lim_{t\to 0^+} \frac{F(t)}{t^{p}}=\infty\,;$$ 
\end{itemize} 
\begin{itemize}
\item[(\aslabel{$\widetilde f_\infty$}) \label{Hp_PS_sQ}]\qquad $q<p$ in \eqref{Hp_estSC}. 
\end{itemize}

Our second main result is the following:

\begin{theorem} \label{Theorem2} Assume the conditions  \eqref{Hp_P2_L1Li}$-$\eqref{Hp_P4_Bbeta}, \eqref{Hp_estSC}, \eqref{Hp_forig_up} and   \eqref{Hp_PS_sQ}. 
Then there exists $a^{\ast}>0$ 
such that, if $a\in[0,a^{\ast})$, 
 problem \eqref{Problem-P} has 
 a positive weak solution $u_a\in C(\mathbb{R}^N) \cap D^{1,p}(\mathbb{R}^N)$. 
\end{theorem}

In the proof of Theorem \ref{Theorem2}, the condition \eqref{Hp_forig_up} will produce a situation where the origin is not a local minimum  for the energy functional, while  \eqref{Hp_PS_sQ} will make the functional coercive, in view of  \eqref{Hp_estSC}.  It will be then possible to obtain solutions via minimization. As in the proof of Theorem \ref{Theorem1}, we will work with a nonsmooth approximating functional that will give us an approximate solution. After some computation, we prove that this approximate solution is in fact a solution for the original problem when $a$ is small enough.

\par\medskip

   \begin{remark} Observe that if  $f,h$ satisfy the  set of conditions of Theorem \ref{Theorem1} or those of Theorem \ref{Theorem2} and $u$ is a solution of Problem \eqref{Problem-P},  then the rescaled function $v=a^{\frac{-1}{q-1}}u$ is a solution of the problem:

 	\begin{equation}\label{Problem-P_resc}
  	\left \{
  	\begin{array}{rclcl}
  	-\Delta_p v & =  & a^{\frac{(q-p)}{q-1}} h(x)(\widetilde f_a(v)-1) & \mbox{in} & \mathbb{R}^N, \\
  	v& > & 0 & \mbox{in} & \mathbb{R}^N, \\
  	\end{array}
  	\right.
  	\end{equation}	
  which then takes the form of Problem \eqref{Problem-P2}, with $\lambda:=a^{\frac{(q-p)}{q-1}}$ and a new nonlinearity $\widetilde f_a(t)=a^{-1}f(a^{\frac{1}{q-1}}t)$, which satisfies the same hipotheses of $f$. In particular, if $f(t)=t^{q-1}$ then $\widetilde f_a\equiv f$.

In the conditions of Theorem \ref{Theorem1}, where  $q>p$, we obtain a solution of Problem \eqref{Problem-P_resc} for suitably small values of $\lambda$, while in the conditions of Theorem \ref{Theorem2}, where $q<p$, solutions are obtained    for suitably large values of $\lambda$.
 
 It is worth noting that, as $a\to0$,  the solutions of Problem \eqref{Problem-P} that we obtain are bounded and converge, up to subsequences,  to a solution of Problem \eqref{Problem-P} with $a=0$ (see Lemma \ref{lemma6}). As a consequence, the corresponding solutions of Problem \eqref{Problem-P_resc} satisfy $v(x)\to \infty$ for every $x\in\R^N$.

Semipositone problems formulated as in \eqref{Problem-P_resc} were considered recently in \cite{CoRQTeh_semipos,PeShSi_semiposCrit}.
 \end{remark}

 \par\medskip

 As a final result, we also show that, by the same technique used to prove the positivity of our solution, it is possible to obtain a version of Hopf's Lemma for the $p$-Laplacian in the whole $\mathbb{R}^N$, see Proposition  \ref{hopf}. A further consequence is the following Liouville-type result:

\begin{proposition}\label{prop_Liou}
Let $N>p>1$ and  $u \in D^{1,p}(\mathbb{R}^N) \cap C_{loc}^{1,\alpha}(\mathbb{R}^N)$ be a solution of problem:
\begin{equation}
\left \{
\begin{array}{rclcl}
-\Delta_p u & =  & g(x)  \hat f(u)& \mbox{in} & \mathbb{R}^N, \\
u& \geq & 0 & \mbox{in} & \mathbb{R}^N, \\
\end{array}
\right.
\end{equation}
where $\hat f,g$ are continuous,   $g(x)>0$  in $\mathbb{R}^N$ and $\hat f(0)=0$ while $\hat f(t)>0$ for $t>0$. If $\liminf_{|x|\to\infty} |x|^{(N-p)/(p-1)}u(x)=0$, then $u\equiv0$.
\end{proposition}

 \subsection{Organization of the article.} 
 This article is organized as follows: in Section \ref{sec_prelim}, we prove the existence of a nonnegative solution, denoted by $u_a$, for a class of approximate problems. In Section \ref{sec_estim}, we establish some properties involving the approximate solution $u_a$. In Section \ref{sec_prfmain}, we prove the Theorems \ref{Theorem1} and \ref{Theorem2} respectively. Finally, in Section \ref{sec_hopf}, we prove the Proposition  \ref{hopf} about Hopf's Lemma for the $p$-Laplacian in the whole $\mathbb{R}^N$.

 \subsection{Notations.} Throughout this paper, the letters $c$, $c_{i}$, $C$, $C_{i}$, $i=1, 2, \ldots, $ denote positive constants which vary from line to line, but are independent of terms that take part in any limit process. Furthermore, we denote the norm of $L^{p}(\Omega)$ for any $p\geq 1$ by $\|\,.\,\|_{p}$. In some places we will use $"\rightarrow"$,   $"\rightharpoonup"$ and $"\stackrel{*}{\rightharpoonup}"$ to denote the strong convergence, weak convergence and weak star convergence, respectively.

\section{Preliminary results}\label{sec_prelim}

In the sequel, we consider the discontinuous function $f_a:\mathbb{R} \longrightarrow\mathbb{R}$  given by
\begin{equation}\label{eq_fa_d}
	f_a(t)=\left \{
	\begin{array}{ccl}
	f(t)-a & \mbox{if}  & t\geq 0, \\
	0 &  & t<0, \\
	\end{array}
	\right.
\end{equation}
and its primitive  
\begin{equation}\label{eq_Fa}
	F_a(t)=\displaystyle\int_{0}^{t} f_{a}(\tau)d\tau=\left \{
	\begin{array}{ccl}
	F(t)-at & \mbox{if}  & t\geq 0, \\
	0& \mbox{if} &  t\leq 0. \\
	\end{array}
	\right.
\end{equation}
A direct computation gives  
\begin{equation}\label{eq_Fa_est}
	-at^+\leq F_a(t)\leq \left \{\begin{array}{ccl}
	F(t) & \mbox{if}  & t\geq 0, \\
	0 & \mbox{if}  & t\leq 0, 
	\end{array}\right.
\end{equation} 
where $t^+=\max\{t,0\}$.

Our intention is to prove the existence of a positive solution for the following auxiliary problem
\begin{equation}\label{Problem-PA}\tag{AP$_a$}
	\left \{
	\begin{array}{rclcl}
	-\Delta_p u & =  & h(x)f_a(u) & \mbox{in} & \mathbb{R}^N, \\
	u& > & 0 & \mbox{in} & \mathbb{R}^N, \\
	\end{array}
	\right.
\end{equation}
 because  such a solution is also a solution of \eqref{Problem-P}. 

 \medskip 
 
Associated with \eqref{Problem-PA}, we have the energy functional $I_a:D^{1,p}(\mathbb{R}^N)\longrightarrow\mathbb{R}$ defined by
\begin{equation*}
	I_a(u)=\frac{1}{p}\int_{\mathbb{R}^N}|\nabla u|^pdx -\int_{\mathbb{R}^N} h(x)F_{a}(u)dx,
\end{equation*}
which is only locally Lipschitz.

 Hereafter, we will endow $D^{1,p}(\mathbb{R}^N)=\left\{u\in L^{p^*}(\mathbb{R}^N);\, \nabla u\in L^p(\mathbb{R}^N,\mathbb{R}^N)\right\} $ with 
 the usual norm
 $$
 \|u\|=\left( \int_{\mathbb{R}^N}|\nabla u|^{p}\,dx \right)^{\frac{1}{p}}.
 $$
 Since the Gagliardo-Nirenberg-Sobolev inequality (see \cite{Evans})
 $$
 \|u\|_{p^*}\leq S_{N,p} \|u\|
 $$
 holds for all $u\in D^{1,p}(\mathbb{R}^N)$ for some constant $S_{N,p}>0$, we have that the embedding 
 \begin{equation} \label{IM}
 	D^{1,p}(\mathbb{R}^N)\hookrightarrow L^{p^*}(\mathbb{R}^N)
 \end{equation}
 is continuous.
The  following Lemma provides us an useful compact embedding for $D^{1,p}(\mathbb{R}^N)$.
\begin{lemma}\label{l1}
 	Assume \eqref{Hp_P2_L1Li}. 
 	 Then, the embedding $D^{1,p}(\mathbb{R}^N)\hookrightarrow L^q_h(\mathbb{R}^N)$ is continuous and compact for every $q\in [1,p^*)$.  
\end{lemma}
\begin{proof}
The continuity is obtained by H\"older inequality, using  \eqref{IM} and \eqref{Hp_P2_L1Li}:
\begin{equation} \label{I1}
\int_{\mathbb{R}^N}h|u|^{q}\,dx \leq \n{h}_{r}\|u\|^q_{p^*}\leq C_h\|u\|^q, \quad \forall u \in D^{1,p}(\mathbb{R}^N), 
\end{equation}
where $r=p^*/(p^*-q)$ is dual to $p^*/q$.

Let $\{u_n\}$ be a sequence in $D^{1,p}(\mathbb{R}^N)$ with  $u_n\rightharpoonup 0\ \mbox{in}\ D^{1,p}(\mathbb{R}^N).$ For each $R>0$, we have the continuous embedding $D^{1,p}(\mathbb{R}^N) \hookrightarrow W^{1,p}(B_R(0))$. Since the embedding $W^{1,p}(B_R(0)) \hookrightarrow L^p(B_R(0))$ is compact, it follows that $D^{1,p}(\mathbb{R}^N) \hookrightarrow L^p(B_R(0))$ is a compact embedding as well. Hence, for some subsequence, still denoted by itself,  
$$
u_n(x)\rightarrow 0\ \mbox{a.e. in}\ \mathbb{R}^N.
$$
By the continuous embedding \eqref{IM}, we also know that  $\{|u_n|^{q}\}$ is a bounded sequence in $L^{\frac{p^*}{q}}(\mathbb{R}^N)$. Then, up to a subsequence if necessary, 
$$
|u_n|^{q}\rightharpoonup 0 \mbox{ in }   L^{\frac{p^*}{q}}(\mathbb{R}^N), 
$$
or equivalently,
$$
\int_{\mathbb{R}^N}|u_n|^{q}\varphi dx \to 0, \quad   \forall \varphi \in L^{r}(\mathbb{R}^N).
$$
As \eqref{Hp_P2_L1Li}  guarantees that $h \in L^{r}(\mathbb{R}^N)$, it follows that 
$$
\int_{\mathbb{R}^N}h(x)|u_n|^qdx \to 0.
$$ 
This shows that $u_n \to 0$ in $L^q_h(\mathbb{R}^N)$, finishing the proof.
\end{proof}

We also give the following result that will be used later.
\begin{lemma} \label{CN} 
If $u_n\rightharpoonup u$ in $D^{1,p}(\mathbb{R}^N)$ and \eqref{Hp_P2_L1Li}
holds, then    
$$
\int_{\mathbb{R}^N}h(x)|u_n-u||u_n|^{q-1}\,dx \to 0 \quad \mbox{as} \quad n \to +\infty, \quad \forall q\in[ 1,p^*).
$$	
\end{lemma}
\begin{proof}
Set $r=\frac{p^*}{q-1}\in \left(\frac{p^*}{p^*-1}, \infty\right]$ and $r'=\frac{p^*}{p^*-(q-1)}\in [1,p^*)$ its dual exponent. 
First note that $\{u_n\}$ is bounded in $D^{1,p}(\mathbb{R}^N)$, and so, $\{|u_n|^{q-1}\}$ is bounded in $L^r(\mathbb{R}^N)$ 
by \eqref{IM}, while $h|u_n-u| \to0$ in $L^{r'}$ 
since we can apply  Lemma \ref{l1} with $h^{r'}$ in the place of $h$, which also satisfies condition \eqref{Hp_P2_L1Li}. Then by H\"older inequality 
  	$$
     	\int_{\mathbb{R}^N}h(x)|u_n-u||u_n|^{q-1}\,dx \leq  \|h(u_n-u)\|_{r'}\|u_n^{q-1}\|_{r}\to 0.
 	$$
\end{proof}

\subsection{Critical points theory for  the functional $I_a$ }
As mentioned in the last subsection, the functional $I_a$ is only locally Lipschitz in $D^{1,p}(\mathbb{R}^N)$, then we cannot use variational methods for $C^1$ functionals. Having this in mind, we will then use the theory of critical points  for locally Lipschitz functions  in a Banach space, see Clarke \cite{Clarke_nonsmooth} for more details. 
 
First of all, we recall that $u\in D^{1,p}(\mathbb{R}^N)$  is a critical point of $I_a$ if  
\begin{equation}\label{eq_varminim}
 \int_{\mathbb{R}^N}|\nabla u|^{p-2}\nabla u \nabla v\,dx+\int_{\mathbb{R}^N}h(x)(-F_a)^0(u,v)\,dx\geq 0, \quad \forall v \in D^{1,p}(\mathbb{R}^N),
\end{equation}
where $$(-F_a)^0(t,s)=\limsup_{\xi\searrow0,\,\tau\to t}\frac {-F_a(\tau+\xi s)+F_a(\tau)}{\xi }$$ indicates the generalized directional derivative of $-F_a$ at the point $t$ along the direction $s$.

It is easy to see that  a global minimum is always  a critical point; moreover, an analogous of the classical mountain pass theorem holds true (see \cite{Chang81_VMnsm}), where a critical point  in the sense of \eqref{eq_varminim} is obtained at the usual minimax level provided the following form of (PS)-condition holds true:
\begin{itemize}
\item [(\aslabel{$PS_L$}) \label{PSL}] 
 If $\{u_n\}$ is a sequence in $D^{1,p}(\mathbb{R}^N)$ such that $\{I_a(u_n)\}$ is bounded and   	    	  	  	   
\begin{equation}\label{eq_varPS}
    \int_{\mathbb{R}^N}|\nabla u_n|^{p-2}\nabla u_n \nabla v\,dx+\int_{\mathbb{R}^N}h(x)(-F_a)^0(u_n,v)\,dx\geq- \varepsilon_n\n{v}, 
\end{equation}
$\forall v \in D^{1,p}(\mathbb{R}^N),$ where $\varepsilon_n\to0$,   then $\{u_n\}$ admits a convergent subsequence. 
\end{itemize}

 \par\medskip
 
In the next Lemma,  let us collect some useful properties that can be derived by the definition of critical points of $I_a$, given in \eqref{eq_varminim}.   
\begin{lemma}\label{lm_prop_minim}
Assume \eqref{Hp_P2_L1Li} 
and (\ref{Hp_estSC}). Then a  critical point $u_a$ of the functional $I_a$, as defined in  \eqref{eq_varminim},   has the following properties:
\begin{enumerate}
\item $u_a\geq 0$ in $\R^N$;   
\item if $u_a>0$ in $\R^N$ then it is a weak solution of problem \eqref{Problem-PA}, and also a solution of problem \eqref{Problem-P};
\item  $u_a$ is a weak subsolution of  $-\Delta_p u=h(x)f(u)$ in $\R^N$;
\item   $u_a$  is a weak supersolution of   $-\Delta_p u=h(x)(f(u)-a)$ in $\R^N$.
\end{enumerate}
\end{lemma}  
\begin{proof}
Straightforward calculations give 
\begin{equation}
  	\label{eq_Fa0}
  	(-F_a)^0(t,s)=\begin{cases}
  	-(f(t)-a)s& \mbox{for $t>0,\ s\in\R$}\\
  	as&\mbox{for $t=0,\,s>0$}\\
  	0& \mbox{for $\begin{cases}t<0,\ s\in\R\\t=0,\,s\leq0\,.\end{cases}$}
  	\end{cases}
\end{equation}
By using $u_a^-=\max\cub{0,-u_a}$ as a test function in \eqref{eq_varminim} we get   
$$-\n {u_a^-}^p=  \int_{\mathbb{R}^N}|\nabla u_a|^{p-2}\nabla u_a \nabla u_a^-\,dx\geq-\int_{\mathbb{R}^N}h(x)(-F_a)^0(u_a,u_a^-)\,dx \geq 0\,,$$
then $u_a^-\equiv0$ and then (1.) is proved.
  
If $u_a>0$ in $supp(\phi)$, then from  \eqref{eq_varminim}, 
\begin{equation*}
  \int_{\mathbb{R}^N}|\nabla u_a|^{p-2}\nabla u_a \nabla \phi\,dx\geq
  -\int_{\mathbb{R}^N}h(x)(-F_a)^0(u_a,\phi)\,dx=\int _{\mathbb{R}^N} h(x)f_a(u_a) \phi\,dx\,,
\end{equation*}
and by testing also with $-\phi$ one obtains equality, then (2.) is proved.

If $\phi\geq0$ in \eqref{eq_varminim} then 
\begin{equation*}
  \int_{\mathbb{R}^N}|\nabla u_a|^{p-2}\nabla u_a \nabla \phi\,dx\geq
  -\int_{\mathbb{R}^N}h(x)(-F_a)^0(u_a,\phi)\,dx\geq \int_{\mathbb{R}^N}h(x)(f(u_a)-a)\phi\,dx\,;
\end{equation*}
by  testing with $-\phi$ one obtains
\begin{equation*}
  - \int_{\mathbb{R}^N}|\nabla u_a|^{p-2}\nabla u_a \nabla \phi\,dx\geq
  -\int_{\mathbb{R}^N}h(x)(-F_a)^0(u_a,-\phi)\,dx\geq -\int_{\mathbb{R}^N}h(x)f(u_a)\phi\,dx\,.
\end{equation*}
The above analysis guarantees that, for every $\phi\in  D^{1,p}(\R^N)$ with $\phi\geq0$, 
\begin{equation}\label{eq_subsup}
  \int_{\mathbb{R}^N}h(x)(f(u_a)-a)\phi\,dx\leq \int_{\mathbb{R}^N}|\nabla u_a|^{p-2}\nabla u_a \nabla \phi\,dx\leq\int_{\mathbb{R}^N}h(x)f(u_a)\phi\,dx \,,
\end{equation} 
which proves the claims (3.) and (4.).  
\end{proof}

\subsection{Mountain pass geometry}\label{sec_MP}
Throughout this subsection we assume the hypotheses of Theorem \ref{Theorem1}. 
The next two Lemmas will be useful to prove that in this case $I_a$ verifies the mountain pass geometry. 	
\begin{lemma}\label{lemma1}
	There exist $\rho,\alpha>0$ such that 
	$$I_a(u)\geq\alpha, \qquad\text{ for $\n u=\rho$ and any $a\geq0$.}$$
\end{lemma}
\begin{proof}
Notice that, in view of \eqref{Hp_forig}, \eqref{Hp_estSC} and \eqref{eq_Fa_est},  given $\epsilon>0$, there exists $C_{\epsilon}>0$ such that  
$$
F_a(t)\leq{\epsilon}|t|^{p}+ C_{\epsilon}{|t|^{q}} ,\quad \forall t\in\R.
$$ 			
Thus, by Lemma \ref{l1},
$$\int_{\R^N}h(x)F_a(u(x))dx\leq {\epsilon}C\|u\|^{p}+ C_{\epsilon}C{\|u\|^{q}}, \quad \forall u\in D^{1,p}(\R^N).$$
Thereby, setting $\n u=\rho$, we obtain
$$I_a(u)\geq \rho^p\rob{\frac1p-\varepsilon  C-{CC_\varepsilon}\rho^{q-p}}\,.$$
Now, fixing $\varepsilon=1/(2pC) $ and choosing  $\rho$ sufficiently small such that $CC_\varepsilon \rho^{q-p}\leq 1/4p$, so that the term in parentheses is at least $1/4p$, we see that   the claim is satisfied by taking  $\alpha =(1/4p)\rho^p$.  
\end{proof}

\begin{lemma}\label{lemma2}
	There exists $v\in D^{1,p}(\mathbb{R}^N) $ and $a_1>0$ such that $\|v\|>\rho$ and $I_a(v)< 0$, for all $a \in [0,a_1)$.
\end{lemma}
\begin{proof}
Fix a function
$$
\varphi\in C_{0}^{\infty}(\mathbb{R}^N) \setminus \{0\},  \quad \mbox{with} \quad \varphi \geq 0  \quad \mbox{and} ~ ||\varphi||=1.
$$
Note that for all $t>0$, 
\begin{eqnarray*}
	I_{a}(t\varphi) &=& \frac1p t^p- \int_{\Omega}h(x)F_a(t\varphi)dx \\
	&=&\frac1p t^p- \int_{\Omega}h(x)F(t\varphi)\,dx + a\int_{\Omega}h(x)t\varphi\, dx,
\end{eqnarray*}
where $\Omega=supp \,\varphi$. Now, estimating with \eqref{AR} and assuming that $a$ is bounded in some set $[0,a_1)$, we find
\begin{eqnarray}\label{eq_est_abv_phi}
I_{a}(t\varphi) & \leq& \frac1p t^p- {A_1t^{\theta}}\int_{\R^N}h(x)\varphi^\theta dx+ B_1\n h_1+t a_1\int_{\R^N}h(x)\varphi dx \,.
\end{eqnarray}
Since $h>0$ the two integrals are positive, and using the fact that $\theta>p>1$, we can fix $t_1>\rho$ large enough so that $I_a(v)<0$, where $v=t_1\varphi\in D^{1,p}(\mathbb{R}^N)$.
\end{proof}

In the sequel, we are going to prove the  version of  (PS)-condition required in the critical points theory for Lipshitz functionals, for the functional  $I_a$. To do this, observe that \eqref{Hp_PS_SQ} yields that $f_a$ also satisfies the famous condition due to Ambrosetti-Rabinowitz, that is,
there exists $T>0$, which does not depend on $a\geq 0$, such that
\begin{equation}\label{ARCondition}
\theta F_a(t) \leq tf_a(t)+T, \quad  t\in\mathbb{R}\,.
\end{equation}

\begin{lemma}\label{lemma3}
For all $a\geq0$, the functional $I_a$ satisfies the condition \eqref{PSL}.
\end{lemma}
\begin{proof}
Observe that, by \eqref{eq_Fa0},  $(-F_a)^0(t,\pm t)= \mp f_a(t)t$ for all $t \in \mathbb{R}$.  Then, from   \eqref{eq_varPS},  
$$
\m{\int_{\mathbb{R}^N}|\nabla u_n|^{p}\,dx-\int_{\mathbb{R}^N}h(x)f_a(u_n) u_n\,dx }\leq \varepsilon_n\n{u_n}.
$$
For $n$ large enough, we assume  $\varepsilon_n<1$ so we get 
\begin{equation}\label{ineq1-lemma1}
   -\|u_n\|-\|u_n\|^{p} \leq-\int_{\mathbb{R}^N} h(x)f_{a}(u_n)u_n dx\,.
\end{equation}
On the other hand, since   $|I_a(u_n)|\leq K$ for some $K>0$, it follows that 
\begin{equation}\label{ineq2-lemma1}
   \frac{1}{p}\|u_n\|^{p} - \int_{\mathbb{R}^N}h(x)F_a(u_n)dx\leq K, \quad \forall n \in\mathbb{N}.
\end{equation}
From \eqref{ARCondition} and \eqref{ineq2-lemma1}, 
\begin{equation}\label{ineq3lemma1}
   \frac{1}{p}\|u_n\|^{p} - \frac{1}{\theta}\int_{\mathbb{R}^N}h(x)f_a(u_n)u_n\,dx-\frac{1}{\theta}T\|h\|_1 \leq K, \quad 
\end{equation}
thereby, by \eqref{ineq1-lemma1} and \eqref{ineq3lemma1},
$$
\left(\frac{1}{p}-\frac{1}{\theta}\right)|| u_n||^p - \frac{1}{\theta}|| u_n||\leq K+\frac{1}{\theta}T\|h\|_1, 
$$
for $n$ large enough. This shows that $\{u_n\}$ is bounded in $D^{1,p}(\mathbb{R}^N)$. Thus, without loss of generality, we may assume that
$$
u_n\rightharpoonup u ~~ \mbox{in} ~~ D^{1,p}(\mathbb{R}^N)
$$
and
$$
u_n(x) \to u(x) \quad \mbox{a.e. in} \quad \mathbb{R}^N.
$$
By \eqref{eq_Fa0} and conditions \eqref{Hp_estSC}-\eqref{Hp_forig}, there exists $C>0$ that does not dependent on $a$ such that
$$
|(-F_a)^0(t,s)|\leq \rob{C(|t|^{q-1}+|t|)+a}|s|
$$
and so, 
$$
|h(x)(-F_a)^0(u_n,u_n-u)| \leq Ch(x)|u_n-u|(|u_n|^{q-1}+|u_n|+a).
$$
By Lemma \ref{CN},  we have the limit 
$$
  	\int_{\mathbb{R}^N} h(x)(-F_{a})^0(u_n,\pm(u_n-u)) dx \to 0,
$$
that combines with the inequalities  below, obtained from \eqref{eq_varPS},
\begin{multline}
 	-\varepsilon_n\n{u-u_n} -\int_{\mathbb{R}^N}h(x)(-F_a)^0(u_n,u-u_n)\,dx \\\leq \int_{\mathbb{R}^N}|\nabla u_n|^{p-2}\nabla u_n \nabla (u-u_n)\,dx\\\leq\int_{\mathbb{R}^N}h(x)(-F_a)^0(u_n,u_n-u)\,dx+ \varepsilon_n\n{u-u_n},
\end{multline}   	   
to give 
\begin{equation}\label{conv1lemma3}
\int_{\mathbb{R}^N}|\nabla u_n|^{p-2}\nabla u_n \nabla (u-u_n)\,dx\to 0.
\end{equation}
The weak convergence $u_n\rightharpoonup u$ in $D^{1,p}(\mathbb{R}^N)$ yields
\begin{equation}\label{conv2lemma3}
\int_{\mathbb{R}^N} |\nabla u|^{p-2}\nabla u\nabla (u_n-u)dx \to 0.
\end{equation}
From \eqref{conv1lemma3},  \eqref{conv2lemma3} and the (S+) property of the $p$-Laplacian, we deduce that  $u_n\to u$ in   $D^{1,p}(\mathbb{R}^N)$, finishing the proof. Here, the Simon's inequality found in \cite[Lemma A.0.5]{PA} plays an important role to conclude the strong convergence. 
\end{proof}

Next, we obtain a critical point for $I_a$, by the mountain pass theorem for Lipschitz functionals. Furthermore, we will make explicit the dependence of the constants on the bounded interval $[0,\overline a)$ where the parameter $a$ is taken, by using as subscript its endpoint, which we still have to fix, while we will not mention their dependence on $h$ and  $f$.
\begin{lemma}\label{lm_ua}
There exists a constant $C_{a_1}>0$	such that $I_a$ has a critical point  $u_a\in D^{1,p}(\R^N)$  satisfying $0<\alpha\leq I_a(u_a)\leq C_{a_1}$, for every  $a\in[0,a_1)$.  
\end{lemma}
\begin{proof}
The Lemmas \ref{lemma1}, \ref{lemma2} and \ref{lemma3} guarantee that we can apply the mountain pass theorem for Lipchitz functionals due to \cite{Chang81_VMnsm} 
to show the existence of a critical point $u_a \in D^{1,p}(\mathbb{R}^N)$ for all $a \in [0,a_1)$, with $I_a(u_a)=d_a\geq \alpha >0$, where $d_a$ is the mountain pass level associated with $I_a$.

Now, taking  $\varphi\in C_{0}^{\infty}(\Omega)$ as in the proof of Lemma \ref{lemma2}, $t>0$, and estimating as in \eqref{eq_est_abv_phi}, we see that    
$I_{a}(t\varphi) $ is bounded from above, uniformly if $a\in[0,a_1)$. Consequently, the mountain pass level is also estimated in the same way, that is, 
$$
0<\alpha\leq d_a =I_a(u_a) \leq \max\{I_a(t\varphi); t\geq 0\} \leq  C_{a_1}.
$$
\end{proof}

The next Lemma establishes a very important estimate involving the Sobo\-lev norm of the solution $u_a$ for $a \in [0,a_1)$.
\begin{lemma}\label{lm_nHincompact}
There exist constants  $k_{a_1},K_{a_1}$, such that $0<k_{a_1}\leq\|u_a\|\leq K_{a_1}$ for all $a \in [0,a_1)$.
\end{lemma}
\begin{proof}
Using again that  $(-F_a)^0(t,\pm t)= \mp f_a(t)t$, we get from \eqref{eq_varminim} that 
\begin{equation}\label{eq_Ipuu}
 \|u_a\|^{p}-\int_{\R^N}h(x) f_{a}(u_a)u_a=0\,.
 \end{equation}
By Lemma \ref{lm_ua}, and subtracting \eqref{eq_Ipuu} divided by $\theta$,
we get the inequality below
$$
	C_{a_1} \geq I_a(u_a)  =  \rob{ \frac 1p-\frac1\theta}  \|u_a\|^{p} + \int_{\R^N}h(x) \left(\frac{1}{\theta}f_{a}(u_a)u_a -F_a(u_a)\right)dx,
$$
which combined with \eqref{ARCondition} leads to 
$$
C_{a_1}\geq\rob{ \frac 1p-\frac1\theta}  \|u_a\|^{p} -\n{h}_\infty T\,,
$$
establishing the estimate from above.

In order to get the estimate from below, just note that by \eqref{eq_Fa_est} and the   embeddings in Lemma \ref{l1},
$$
\alpha\leq I_a(u_a)\leq 
\frac1p\n{u_a}^p+a\int_{\R^N} u_a^+\,dx \leq \frac1p\n{u_a}^p+Ca_1\n{u_a}.
$$
This gives the desired estimate from below.  
\end{proof}

\subsection{Gobal minimum geometry}\label{sec_min}
Throughout this subsection, we assume the hypotheses of Theorem \ref{Theorem2}. 
The next three Lemmas will  prove that $I_a$ has a global minimum at a negative level.
  
\begin{lemma}\label{lemma1m}
There exist $a_1,\alpha>0$ and  $u_0\in D^{1,p}(\mathbb{R}^N)$  
such that 
$$I_a(u_0)\leq-\alpha, \qquad\text{ for  any $a\in[0,a_1)$.}$$
\end{lemma}
\begin{proof}
Let $\varphi\in C_{0}^{\infty}(\mathbb{R}^N)$ be as in the proof of Lemma \ref{lemma2}. For $t>0$, 
\begin{eqnarray*}
	I_{a}(t\varphi)	&=&\frac1p t^p - \int_{\Omega}h(x)F(t\varphi)\,dx + a\int_{\Omega}h(x)t\varphi\, dx\,,
\end{eqnarray*}
where $\Omega=supp \,\varphi$.

From \eqref{Hp_forig_up} and using the fact that $\displaystyle \inf_{x \in supp \,\varphi }h(x)=h_0>0$,  we have that, for $t_0>0$ small enough 
$$  \quad\int_{\Omega}h(x)F(t_0\varphi)\,dx\geq  \frac2pt_0^{p}\,.$$ 
Therefore, 
$$
I_a(t_0\varphi)\leq -\frac{1}p t_0^{p}+at_0\int_{\Omega}h(x) \varphi\,dx.
$$
Now fixing $\alpha=\frac{1}{2p}t_0^p>0$ and  choosing  $a_1=a_1(t_0)$ in such way that \\$a_1t_0\int_{\Omega} h(x)\varphi\,dx\leq\alpha$, we derive that
$$
I_a(t_0\varphi)\leq -\alpha<0 \qquad \text{for $a\in[0,a_1)$,}
$$
showing the Lemma. 
\end{proof}

\begin{lemma}\label{lemma2m}
 	$I_a$ is coercive, uniformly with respect to $a\geq0$, in fact, there exist $H,\rho>0$ independent of $a$ such that $I_a(u)\geq H$ whenever $\n u\geq \rho$.
\end{lemma}
\begin{proof} 
By \eqref{Hp_estSC} and Lemma \ref{l1}, there is $C>0$ such that 
\begin{eqnarray}\label{eq_coerc}
I_a(u)&\geq &\frac 1p\n{u}^{p } - \int_{\R^N}h(x)\rob{C+C|u|^q}\,dx
\\\nonumber&\geq &\frac 1p\n{u}^{p }-C-C\n{u}^q,
\end{eqnarray}
then the claim follows easily since $p>q$ from \eqref{Hp_PS_sQ}.
\end{proof}

\begin{lemma}
For every $a\in\R$, $I_a$ is weakly lower semicontinuos.
\end{lemma}
\begin{proof}
The proof is classical, since the norm is weakly lower semicontinuos and
the term $\int_{\R}h(x)F_a(u)\,dx$ is weakly continuous. To see this, let $\{u_n\}$ be a sequence in $D^{1,p}(\mathbb{R}^N)$ such that 
$$
u_n\rightharpoonup u ~~ \mbox{in} ~~ D^{1,p}(\mathbb{R}^N).
$$
Then, proceeding as in the proof of Lemma \ref{l1}, up to a subsequence 
$$
u_n(x) \to u(x) \quad \mbox{in $L^{q}_h(\R^N)$\ and\  a.e. in} \ \mathbb{R}^N.
$$
This means that $w_n=h^{1/q}u_n\to w= h^{1/q}u$ in $L^{q}$, as a consequence, we may also assume that $\{w_n\}$ is dominated by some $g\in L^{q}$. On the other hand, by \eqref{Hp_estSC} and \eqref{Hp_P2_L1Li}
$$
|h\,F_a(u_n)|\leq h\,C(|u_n|^q+1)\leq C(g^q+h)\in L^1(\mathbb{R}^N),
$$ 
and so,  $hF_a(u_n)$  is dominated and converges to its a.e. limit $hF_a(u)$. Since the same argument can be applied to any subsequence of the initial sequence, we can ensure that  $$
\lim_{n \to +\infty}\int_{\R^N} hF_a(u_n)\,dx=\int_{\R^N} hF_a(u)\,dx
$$ along any $D^{1,p}$-weakly convergent sequence. 
\end{proof}

We will now obtain a candidate solution for problem \eqref{Problem-PA} by minimization.
\begin{lemma}\label{lm_ua_m}
There exists a constant $C_{a_1}>0$	such that
$I_a$ has a global minimizer 
$u_a\in D^{1,p}(\R^N)$ satisfying  $0>-\alpha\geq I_a(u_a)\geq -C_{a_1}$, for every  $a\in[0,a_1)$.  
\end{lemma}
\begin{proof}
The minimizer is obtained in view of the above Lemmas. Actually the global minimum of $I_a$ stays below $-\alpha$  by Lemma \ref{lemma1m}, while the boundedness from below is a consequence of \eqref{eq_coerc}. 
\end{proof}

The next Lemma establishes the same important estimate as the one in Lemma \ref{lm_nHincompact}, for the minimizer $u_a$.
\begin{lemma}\label{lm_nHincompact_m}
There exist constants  $k_{a_1},K_{a_1}$, such that $0<k_{a_1}\leq\|u_a\|\leq K_{a_1}$ for all $a \in [0,a_1)$.
\end{lemma}
\begin{proof} 
The bound from above  for the norm of $u_a$ is a consequence of the uniform coercivity proved in Lemma \ref{lemma2m}, since $I_a(u_0)<0$. For the estimate from below, just note that by \eqref{eq_Fa_est},   
$$
0>-\alpha\geq I_a(u_a)=
\frac1p\n u^p-\int_{\R^N} h(x)F_a(u_a)\,dx  
$$
$$
\geq-\int_{\R^N} h(x)F( u_a^+)\,dx 
$$
and the right hand side goes to zero if $\n{u_a}$ goes to zero, by Lemma \ref{l1} and the continuity of the integral.
\end{proof}

\section{Further estimates for the critical points  $u_a$}\label{sec_estim}
From now on $u_a$ will be the critical point obtained  in Lemma \ref{lm_ua}  or in Lemma \ref{lm_ua_m}. Our first result ensures that the family $\{u_a\}_{a\in[0,\overline a)}$ is a bounded set in $L^{\infty}(\R^N)$ for $\overline a$ small enough. This fact is crucial in our approach.
\begin{lemma} \label{Estimativa} 
There exists $C_{a_1}^\infty>0$  such that
\begin{equation}\label{eq_estCinfty}
   \|u_a\|_{\infty} \leq C_{a_1}^\infty, \quad \forall a \in [0,a_1).
\end{equation}
\end{lemma}
\begin{proof} 
By Lemma \ref{lm_prop_minim} we know that for $a\in[0,a_1)$,  $u_a\geq0$ and it is a weak subsolution of 
$$
-\Delta_p u=h(x)f(u), \quad \mbox{in} \quad \mathbb{R}^N.
$$
  
In the case of the mountain pass  geometry, $u_a$ is also 
a weak subsolution of 
$$
-\Delta_p u=h(x)\alpha(x)\rob{1+|u|^{p-2}u}\quad \mbox{in} \quad \mathbb{R}^N,
$$
where, from  \eqref{Hp_estSC}  and  \eqref{Hp_PS_SQ}, 
$$
\alpha(x):=\frac{f(u_a(x))}{1+u_a(x)^{p-1}}\leq D(1+u_a(x)^{q-p})\quad \mbox{in} \quad \mathbb{R}^N,
$$ 
for some $D>0$ which depends  only  on $f$.

Let $K_\rho(x)$ denote a cube centered at $x$ with edge length $\rho$, and $\n{\cdot}_{r,K}$ denote the $L^r$ norm restricted to the set $K$. Our goal is to prove that, for a fixed $\rho>0$ and any $x\in\R^N$, one  has 
\begin{equation}\label{eq_Tr1} 
   \sup_{K_\rho(x)}u_a\leq C \rob{1+\n{u_a}_{p^*, K_{2\rho}(x)}}
\end{equation} 
where $C$ depends on $p,N,f,h $ only.  
Since   $K_\rho(x)$ can be taken anywhere and the right hand side is bounded for $\cub{u_a}_{a\in[0,a_1)}$ by Lemma \ref{lm_nHincompact} and \eqref{IM}, equation \eqref{eq_Tr1} gives  a uniform bound for $u_a$ in $L^\infty$, proving  our claim.
       
In order to prove \eqref{eq_Tr1},   we will use   Trudinger \cite[Theorem 5.1]{Trud67_HarnType}   (see also Theorem 1.3 and Corollary 1.1). For this,  we need to show that  
$$
\sup_{x\in\R^N,\ \rho>0}\frac{\n {h\alpha} _{N/p,K_\rho(x)}}{\rho^\delta}\leq C
$$
for a suitable $\delta>0$  and $C$ that do not depend on $a\in[0,a_1)$ (see eq. (5.1) in \cite{Trud67_HarnType}).

Actually let $\tau =p^*/(q-p)>N/p$, then  
$$
  \n {h\alpha} _{N/p,K_\rho(x)} \leq\n{h\alpha}_{\tau,{K_\rho(x)}} |{K_\rho(x)}|^{p/N-1/\tau}\,,
$$ 
and 
\begin{multline}
\n {h\alpha} _{\tau,K_\rho(x)}^\tau=  \int_{K_\rho(x)} (h\alpha) ^\tau dx \leq\int_{K_\rho(x)}  h^\tau D(1+u_a^{q-p})^\tau\,dx\leq \\\leq D'\int_{K_\rho(x)}  h^\tau (1+u_a^{p^*})dx\leq D''(1+\n{u_a}_{p^*}^{p^*} )\,.
\end{multline}    
Using the fact that $|{K_\rho(x)}|= \rho^N$,  we conclude that    $ \rho^{-\delta}\n {h\alpha} _{N/p,K_\rho(x)}$ is bounded, for a suitable $\delta>0$, by a constant depending only on $p,N,f,h $ and $\n{u_a}_{p^*}$, which is bounded by Lemma \ref{lm_nHincompact} and \eqref{IM}. 

In the case of the minimum geometry,  we can take $\alpha$ to be a constant and then the boundedness of $ \rho^{-\delta}\n {h\alpha} _{N/p,K_\rho(x)}$ is easily obtained since $h\in L^\infty$ (in this case \eqref{eq_Tr1}  can also be obtained directly from Theorem 1.3 and Corollary 1.1  in \cite{Trud67_HarnType}).
\end{proof}
 
In what follows, we show an estimate from below of the norm $L^{\infty}(B_\gamma)$ of $u_a$ for $a$ small enough, where $B_\gamma \subset \mathbb{R}^N$ is the open ball centered at origin with radio $\gamma>0$. This estimate is a key point to understand the behavior of the family $\{u_a\}$ when $a$ goes to 0.

\begin{lemma}\label{lemma6} There exist $\delta,\gamma>0$ that do not depend on $a \in [0,a_1)$, such that $\|u_a\|_{\infty,{B_\gamma}}\geq \delta$  for all $a\in[0,{a}_1)$.
\end{lemma}
\begin{proof}
By \eqref{eq_subsup}, since $u_a\geq0$,
\begin{equation}
\int_{\mathbb{R}^N}|\nabla u_a|^{p}\,dx\leq\int_{\mathbb{R}^N}h(x)f(u_a)u_a\,dx \,.
\end{equation} 
By Lemma \ref{lm_nHincompact} (resp. Lemma \ref{lm_nHincompact_m})  the left hand side is bounded from below by $k_{a_1}^p$.
Let now 
\\\indent $\bullet$  
$\Gamma$ be such that $f(t)t< \Gamma$ for $t\in [0,C_{a_1}^\infty]$, where $C_{a_1}^\infty$ was given in Lemma \ref{Estimativa},
\\\indent $\bullet$  
$\gamma$ be such that $\int_{\R^N\setminus B_\gamma}h\,dx<k_{a_1}^p/(2\Gamma)$,
\\\indent $\bullet$    
$\delta$ be such that $f(t)t< k_{a_1}^p/(2\n{h}_\infty|B_\gamma|)$ for $t\in [0,\delta]$.
\\
Then if $u_a<\delta$  in $B_\gamma$ we are lead to the  contradiction 
$$
k_{a_1}^p \leq \int_{\R^N}|\nabla u_a|^p \,dx \leq\int_{\R^N\setminus B_\gamma} h(x)f(u_a)\,u_a\,dx+\int_{B_\gamma} h(x)f(u_a)\,u_a\,dx<k_{a_1}^p
$$
and then the claim is proved.
\end{proof}
 
 We can now prove the following convergence result.
\begin{lemma}\label{lm_subtou}
Given  a sequence of positive numbers $a_j\to 0$, there exists  $u\in D^{1,p}(\R^N)$ and $\beta>0$ such that, up to a subsequence,  $u_{a_j}\to u$ weakly in $D^{1,p}(\R^N)$  and  in ${C}^{1,\beta}$ sense in compact sets. Moreover, $u>0$  is a solution of \eqref{Problem-P} with $a=0$.
\end{lemma}
\begin{proof}
Fixing $u_j=u_{a_j}$, it follows that $\{u_j\}$ is bounded in $L^\infty(\mathbb{R}^N)$, which means that  we may apply \cite[Theorem 1]{Tolksdorf} to obtain that it is also bounded in ${C}_{loc}^{1,\alpha}(\R^N)$ for some $\alpha>0$. As a consequence, for $\beta\in(0,\alpha)$,
in any compact set $\overline\Omega$ it admits a  subsequence that converges in ${C}^{1,\beta}(\overline\Omega)$ and using a diagonal procedure we see that there exists $u\in {C}^{1,\beta}(\R^N)$ such that, again up to a subsequence,  $u_n\rightarrow u$ in ${C}^{1,\beta}$ sense in
compact sets. From Lemma \ref{lemma6}, $u$ is not identically zero. The boundedness in $W_{loc}^{1,p}(\mathbb{R}^N)$ implies that we may also assume  that $u_j\rightharpoonup u$  in $W_{loc}^{1,p}$ and  in $L_{loc}^{p^*}(\mathbb{R}^N)$.

For $\phi\geq 0$ with  support in some bounded set $\Omega$, from \eqref{eq_subsup}  we have 
\begin{equation}
\int_{\Omega}h(x)(f(u_j)-a_j)\phi\,dx\leq \int_{\Omega}|\nabla u_j|^{p-2}\nabla u_j \nabla \phi\,dx\leq\int_{\Omega}h(x)f(u_j)\phi\,dx \,;
\end{equation} 
 the above convergences bring 
\begin{equation}
 \int_{\Omega}|\nabla u|^{p-2}\nabla u \nabla \phi\,dx=\int_{\Omega}h(x)f(u)\phi\,dx \,,
\end{equation} 
then $u$ is a nontrivial solution of \eqref{Problem-P} with $a=0$, and since $f\geq0$, it follows that $u$ is everywhere positive.
\end{proof}

 \section{Proof of the main Theorems }\label{sec_prfmain}
In order to prove that $u_a>0$ for $a$ small enough,  we will first construct a subsolution that will be used for comparison. 

\begin{lemma}\label{lm_z}

Let $\vartheta>N>p$ be as in  
\eqref{Hp_P4_Bbeta}. 
Given $A,r>0$ there exists $H>0$  such that 
the problem 
\begin{equation}\label{eq_probz}
	\left \{
	\begin{array}{rclcl}
	-\Delta_p z & =  & A & \mbox{in} & B_r\,, \\
	-\Delta_p z & =  & -H |x|^{-\vartheta} & \mbox{in} &  \mathbb{R}^N\setminus B_r\,, \\
	\end{array}
	\right.
\end{equation}
has an explicit family of bounded radial and radially decreasing weak solutions, defined up to an additive constant. More precisely, if we take $H=A\frac {\vartheta-N}{N} r^{\vartheta}$ and  if we fix $\lim_{|x|\to\infty}z(x)=0$, then the solution is
\begin{equation}\label{eq_z} 
	z(x)=  \begin{cases}
	C-\rob{\frac A N}^{1/(p-1)}\frac{p-1}p|x|^{p/(p-1)}& \mbox{ for $|x|<r$}\,,\\
	\rob{\frac {A} {N}r^{\vartheta}}^{1/(p-1)}\frac{p-1}{\vartheta-p}|x|^{(p-\vartheta)/(p-1)} & \mbox{ for $|x|\geq r$}\,,\\
	\end{cases}
\end{equation}
where C is chosen so that the two formulas coincide for $|x|=r$.

\end{lemma}
 
\begin{proof}
For a radial function $u(x)=v(|x|)$ one has $$\Delta_p u =|v'|^{p-2}\sqb{(p-1)v''+\frac{ N-1}{\rho}v'}\,.$$

By substitution, one can see that a function in the form   $u(x)=v(|x|)=\sigma |x|^\lambda$ is a solution of  the equation $\Delta_p u=\varrho|x|^b $ provided
$$
\begin{cases}
\lambda=\frac{p+b}{p-1}\,,
\\ 
|\sigma|^{p-2}\sigma=\frac 1{(N+b)|\lambda|^{p-2}\lambda}\,\varrho\,.
\end{cases}
$$

In particular,
\begin{itemize}
\item if $b=0$  then $\lambda=\frac p{p-1}>0$ and $\sigma$ has the same sign of $\varrho$:
\item   if $b=-\vartheta$ with $\vartheta>N>p$  then $\lambda=\frac {p-\vartheta}{p-1}<0$ and still $\sigma$ has the same sign of $\varrho$.
\end{itemize}
 
Now, taking the two functions 
$$
	\begin{cases}
	-\rob{\frac A N}^{1/(p-1)}\frac{p-1}p|x|^{p/(p-1)}& \mbox{for $|x|<r$\,, }
	\\\rob{\frac {H} {\vartheta-N}}^{1/(p-1)}\frac{p-1}{\vartheta-p}|x|^{(p-\vartheta)/(p-1)}&  \mbox{for $|x|>r$\,, }
	\end{cases}  
$$
they satisfy \eqref{eq_probz} and their radial derivatives are 
$$
	\begin{cases}
	 -\rob{\frac A N}^{1/(p-1)}|x|^{1/(p-1)}& \mbox{for $|x|<r$\,, }
	 \\-\rob{\frac {H} {\vartheta-N}}^{1/(p-1)}|x|^{(1-\vartheta)/(p-1)}&  \mbox{for $|x|>r$. }
	\end{cases}  
$$
Note that the derivatives are equal at $|x|=r$ provided
$$
\frac A N r=\frac {H}{\vartheta-N} r^{1-\vartheta}\,.
$$

Having this in mind, we can therefore construct $z$ piecewise as in \eqref{eq_z} obtaining a bounded, radial and radially decreasing function, to which  any constant can be added.
\end{proof}
 
In order to finalize the proof of our main Theorems, we only  need to show that for any sequence of positive numbers $a_j\to0$ there exists a subsequence of the corresponding critical points  $u_j:=u_{a_j}$ that are positive: this is proved in the following Lemma.
\begin{lemma}\label{lm_ujpos}
If $h$ satisfies \eqref{Hp_P2_L1Li}$-$\eqref{Hp_P4_Bbeta}, then the sequence $\{u_j\}$ satisfies $u_j>0$ for $j$ large enough.
\end{lemma}
\begin{proof}
Fix $r>0$. From Lemma \ref{lm_subtou}, up to a subsequence,  $u_j\to u>0$ uniformly in $\overline{B_r}$. Thus, there exist $A,j_0>0$ such that, 
\begin{equation}\label{eq_hfA}
  \mbox{ $u_j>0$ \quad and \quad  $h(x)f_{a_j}(u_j)\geq A$,\quad  in $\overline{B_r}$,\quad for $j>j_0$. }
\end{equation}
Now let $B$ be the constant in condition  \eqref{Hp_P4_Bbeta},  $ H=A\frac {\vartheta-N}{N} r^{\vartheta}$ as from Lemma \ref{lm_z} and let $j_1>j_0$ be such that $a_jB<H$ for $j>j_1$. Then we have,
\begin{equation}\label{eq_hfB}
	\mbox{   $h(x)f_{a_j}(u_j)\geq-h(x)a_j\geq
	-H |x|^{-\vartheta}$,\ \ in $B_r^C$,\ \   for $j>j_1$. }
\end{equation}
Combining equation \eqref{eq_subsup}, the above inequalities and \eqref{eq_probz}, we get,  for every $\phi\in D^{1,p}(\R^N), \ \phi\geq0$, 
\begin{equation}\label{eq_compuz}
	\int_{\R^N} |\nabla u_j|^{p-2}\nabla u_j\nabla \phi\,dx \geq 
	\int_{\R^N} h(x)(f_{a_j}(u_j))\phi\,dx
	\geq
	\int_{\R^N} |\nabla z|^{p-2}\nabla z\nabla \phi\,dx\,.
\end{equation}

In order to conclude, fix an arbitrary $R>r$ and define  $z_R$ by subtracting from  $z$ a constant  so that $z_R=0$ on $\partial B_R$ (but observe that  $z_R>0$ in $ B_R$).
By \eqref{eq_compuz} (which holds true also for $z_R$) and since $u_j\geq 0=z_R$ on $\partial B_R$ we obtain by Comparison Principle that  $u_j\geq z_R>0$ in $B_R$.
Since $R$ is arbitrary, we have proved that $u_j>0$ in $\R^N$ for $j>j_1$.

In particular, since $z_R\to z$ uniformly as  $R\to \infty$, we conclude that   $u_j\geq z$ in $\R^N$ for $j>j_1$.

\end{proof}
 
\section{Final comments. }\label{sec_hopf}
In this section we would like to point out some results that can be useful when studying problems that involve the $p$-Laplacian operator in the whole $\mathbb{R}^N$. Proposition \ref{hopf} below works like a Hopf's Lemma for the $p$-Laplacian operator in the whole $\mathbb{R}^N$  and it allows to prove the Liouville-type result in Proposition \ref{prop_Liou}.

\begin{lemma}\label{lm_z_h}
Let $N>p$ and $A,r>0$. Then,  the problem 
\begin{equation}\label{eq_probz_h}
	\left \{
	\begin{array}{rclcl}
	-\Delta_p z & =  & A & \mbox{in} & B_r, \\
	-\Delta_p z & =  & 0 & \mbox{in} &  \mathbb{R}^N\setminus B_r, \\
	\end{array}
	\right.
\end{equation}
has an explicit family of bounded radial and radially decreasing weak  solutions, defined up to an additive constant. More precisely, if we fix $\displaystyle \lim_{|x|\to\infty}z(x)=0$, then the solution is
$$
    z(x)=  \begin{cases}
      C-\rob{\frac A N}^{1/(p-1)}\frac{p-1}p|x|^{p/(p-1)}& \mbox{ for $|x|<r$}\,,\\
   \rob{\frac A N}^{1/(p-1)}\frac{p-1}{N-1}r^{N/(p-1)}|x|^{(p-N)/(p-1) }& \mbox{ for $|x|\geq r$}\,,\\
      \end{cases}
$$ 
where C is chosen so that the two formulas coincide for $|x|=r$.
\end{lemma}

\begin{proof}
%
For $|x|>r$ we consider the family of p-harmonic functions $C_1|x|^{\frac{p-N}{p-1}}$, with radial  derivative  $-C_1\frac{N-p}{p-1}|x|^{-\frac{N-1}{p-1}}$.

Then we only have to set  $-\rob{\frac A N}^{1/(p-1)}r^{1/(p-1)}=-C_1\frac{N-p}{p-1}r^{-(N-1)/(p-1)} $, that is,  
\begin{equation}\label{eq_consthopf}C_1=\rob{\frac A N}^{1/(p-1)}\frac{p-1}{N-p}r^{N/(p-1)}.
\end{equation}
\end{proof}

Proceeding as in the proof of Lemma \ref{lm_ujpos}, we obtain the following Proposition as an immediate consequence of the above Lemma.
\begin{proposition}[Hopf's Lemma] \label{hopf}
Suppose $N>p$, $A,r,\alpha>0$ and $u \in D^{1,p}(\mathbb{R}^N) \cap C_{loc}^{1,\alpha}(\mathbb{R}^N)$ satisfying
$$
  \begin{cases}
  -\Delta_p u\geq A>0& in\ B_r\,, 
\\-\Delta_p u\geq 0 & in\ \R^N\,, 
\\u\geq 0& in\ \R^N\,.
\end{cases}
$$
Then $u(x)\geq C|x|^{(p-N)/(p-1) }$ for $|x|>r$, where $C$ is given in \eqref{eq_consthopf}.
\end{proposition}

The above Proposition complements,  in some sense, the study made in \cite[Theorem 3.1]{CTT}, which obtained a similar estimate when $u$ is a solution of a particular class of $p$-Laplacian problems in the whole $\mathbb{R}^N$. 

\begin{remark}\label{rm_udec}
Proposition \ref{hopf} applies in particular to the limit solution $u$ obtained in Lemma \ref{lm_subtou}, providing us with its decay rate at infinity.
\end{remark}

Finally, from Proposition \ref{hopf} it is straightforward to derive Proposition \ref{prop_Liou}.

%
%
%
%
%
%
%
%
%
%
%
%
%
%

\providecommand{\bysame}{\leavevmode\hbox to3em{\hrulefill}\thinspace}

%
%
\end{document}